\documentclass[11pt]{article}
\usepackage{color}
\usepackage{epsfig}

\newcommand{\Rp}{\R_{\geq 0}} 

\newcommand{\tdd}{\mathrm{RC}}
\newcommand{\tDD}{\mathrm{RD}}
\newcommand{\td}{\mathrm{rc}}
\newcommand{\tdD}{\mathrm{rd}}
\newcommand{\lbb}{\nu}
\newcommand{\diag}{\mathrm{diag}}
\newcommand{\Tr}{\mathrm{Tr}}
\newcommand{\G}{{\cal {G}}}
\newcommand{\K}{{\cal {K}}}
\newcommand{\T}[1]{{#1}^T}

\usepackage{
amsmath,amssymb}
\usepackage{
            theorem,%
            amsfonts,%
            pifont}%

\newcommand{\Tau}{{\cal T}}

\newcommand{\Id}{\mathrm{Id}}

\newcommand{\mR}{\mathbf R}
\newcommand{\mRp}{\mathbf R_{\geq 0}}

\renewcommand{\r}[1]{(\ref{#1})}
\newcommand{\R}{\mathbf{R}}
\newcommand{\N}{\mathbf{N}}
\newcommand{\Z}{\mathbf{Z}}

\newcommand{\C}{\mathbf{C}}
\newcommand{\beq}{\begin{equation}}
\newcommand{\eeq}{\end{equation}}
\newcommand{\be}{\begin{equation}}
\newcommand{\ee}{\end{equation}}
\newcommand{\bea}{\begin{eqnarray}}
\newcommand{\eea}{\end{eqnarray}}
\newcommand{\br}{\begin{eqnarray}}
\newcommand{\er}{\end{eqnarray}}
\newcommand{\brs}{\begin{eqnarray*}}
\newcommand{\ers}{\end{eqnarray*}}
\newcommand{\ba}{\begin{array}}
\newcommand{\ea}{\end{array}}

\newcommand{\caG}{{\cal {G}}}

\newcommand{\om}{\omega}
\newcommand{\Om}{\Omega}
\newcommand{\Ga}{\Gamma}
\newcommand{\lb}{\lambda}

\newcommand{\ga}{\gamma}
\newcommand{\tht}{\theta}

\newcommand{\al}{\alpha}

\newcommand{\lp}{\left(}
\newcommand{\rp}{\right)}

\newcommand{\bed}{\begin{description}}
\newcommand{\eed}{\end{description}}

\def\n{\noindent}
\def\rref#1{(\ref{#1})}
\def\eps{\varepsilon}

\def\al{\alpha}

\def\th{\theta}

\def\mt{\mapsto}
\def\EOP{\ \hfill \rule{0.5em}{0.5em} }

\newenvironment{proof}{\noindent {\em Proof. }}{\hfill $\blacksquare$ \vskip 3pt}

\newtheorem{thm}{Theorem}[section]
\newtheorem{lemma}[thm]{Lemma}
\newtheorem{cor}[thm]{Corollary}
\newtheorem{prop}[thm]{Proposition}
\begin{theorembodyfont}{\rmfamily}
\newtheorem{rmk}[thm]{Remark}

\newtheorem{conjecture}{Open problem}
\end{theorembodyfont}

\newtheorem{deff}[thm]{Definition}

\begin{document}

\title{\bf \Large \vspace{5mm}%
On the stabilization of persistently excited linear systems\thanks{The work 
was in part carried out while the first author
was working as Marie Curie Fellow at the Department of Mathematics and Statistics,
University of Kuopio, Finland, supported by the European Commission 6th framework
program ``Transfer of Knowledge" through the project
``Parametrization in the Control of Dynamic Systems'' (PARAMCOSYS, MTKD-CT-2004-509223). The second author was partially supported by  
the grant ArHyCo of the {\it Agence Nationale de la Recherche}. 
}}

\author{Yacine Chitour\footnote{{\indent  Laboratoire des Signaux et
Syst\`emes, Sup\'elec,
 3, Rue Joliot Curie,  91192 Gif s/Yvette, France
  and Universit\'e Paris Sud, Orsay}, {\tt chitour@lss.supelec.fr}}
 \qquad
                Mario Sigalotti\footnote{{\indent
Institut
\'Elie Cartan, UMR 7502 INRIA/Nancy-Universit\'e/CNRS,
 BP 239, Vand\oe uvre-l\`es-Nancy 54506,  France
 and CORIDA, INRIA Nancy -- Grand Est, France},
 {\tt mario.sigalotti@inria.fr}}
}

\date{}
\maketitle

\begin{abstract}
We consider control systems of the type $\dot x = A x +\alpha(t)bu$,
where $u\in\R$, $(A,b)$ is a controllable pair and $\alpha$ is an
unknown time-varying signal with values in $[0,1]$ satisfying a
persistent excitation condition i.e., $\int_t^{t+T}\al(s)ds\geq \mu$ for
every $t\geq 0$, with $0<\mu\leq T$ independent on $t$. We prove
that such a system is stabilizable with a linear feedback depending
only on the pair $(T,\mu)$ if the eigenvalues of
$A$ have non-positive real part. We also show that stabilizability does not hold 
for arbitrary matrices $A$. 
Moreover, the question of whether the system can be stabilized or
not with an arbitrarily large rate of convergence gives rise to a
bifurcation phenomenon in dependence of the parameter $\mu/T$.
\end{abstract}

\section{Introduction}
The present paper is a continuation of \cite{MCSS}, where the study
of 
control linear systems subject to scalar
persistently excited PE-signals was initiated. The general form of such 
systems is given  by
\begin{equation}\label{system}
\dot x = A x +\alpha(t)Bu\,,
\end{equation}
where $x\in \R^n$, $u\in \R^m$ and the function $\alpha$ is a {\em scalar} PE-signal, i.e.,
$\alpha$ takes values in $[0,1]$ and
there exist two positive constants $\mu,T$ such that, for every
$t\geq 0$,
\begin{equation}
\label{Tmu} \int_t^{t+T}\alpha(s)ds\geq \mu.
\end{equation}
Given two positive real numbers $\mu$ and $T$ with $\mu\leq T$, we use ${\cal
{G}}(T,\mu)$ to denote the class of all PE signals verifying
\rref{Tmu}.

In \rref{system}, the PE-signal $\alpha$ can be seen as an input
perturbation modeling the fact that the instants where the control
$u$ acts on the system are not exactly known. If $\alpha$ only takes
the values $0$ and $1$, then \rref{system} actually switches between
the uncontrolled system $\dot x=Ax$ and the controlled one $\dot x =
Ax+Bu$. In that context, the PE condition \rref{Tmu} is designed to
guarantee some action on the system. (For a more detailed discussion on
the interpretation of
persistently excited systems and on the related literature, see \cite{MCSS}.)

Our main concern will be the global asymptotic stabilization of system \rref{system}
with a constant linear feedback $u=-Kx$ where the gain matrix $K$ is
required to be the same \emph{for all} signals in the considered class
${\cal {G}}(T,\mu)$ i.e., $K$ depends only on $A,b,T,\mu$ and not on
a specific element of ${\cal {G}}(T,\mu)$. We refer to such a gain
matrix $K$ as a $(T,\mu)$-stabilizer. It is clear that $(A,B)$ must
be stabilizable for
hoping that a $(T,\mu)$-stabilizer exists
and we will suppose that
throughout the paper. Moreover, 
the stabilizability analysis can be reduced to the controllability subspace and thus to the case
where $(A,B)$ is
controllable.

The questions studied in this paper find their origin in a problem
stemming from identification and adaptive control (cf.
\cite{ABJKKMPR}). Such a problem deals with the linear system $\dot
x = -P(t)u$, where the matrix $P(\cdot)$ is
symmetric non-negative definite and 
plays the role of $\alpha$. If $P\equiv I$, then $u^*=x$ trivially
stabilizes the system exponentially. But what if $P(t)$ is only
semi-positive definite for all $t$? Under which conditions on $P$
does $u^*=x$ still stabilize the system? The answer for this
particular case, can be found in the seminal paper \cite{MORNAR}
which asserts that, if $x\in\mR^n$  and $P\geq0$ is bounded and has
bounded derivative, it is {\em necessary and sufficient}, for the
global exponential stability of $\dot x=-P(t)x$, that $P$ is also
{\em persistently exciting}, 
{\em i.e.}, that there exist
$\mu,T>0$ such that
\begin{equation}\label{1002}
\int_{t}^{t+T} \xi^T P(s)\xi ds \geq \mu,
\end{equation}
for all unitary vectors $\xi\in\mR^n$ and all $t\geq 0$. Therefore,
as regards the stabilization of \rref{system}, the notion of {\em
persistent excitation} seems to be a reasonable additional
assumption 
on
the signals $\alpha$.

Let us recall the main results of \cite{MCSS}. We first addressed
the issue of controllability of \rref{system}, uniformly with
respect to $\alpha\in {\cal {G}}(T,\mu)$. We proved that, if the
pair $(A,B)$ is controllable, then \rref{system} is (completely)
controllable in time $t$ if and only if $t>T-\mu$. We next focused
on the existence of $(T,\mu)$-stabilizers.
We first treated the case where $A$ is neutrally stable and we
showed that in this case the gain $K=B^T$ is a $(T,\mu)$-stabilizer
for system~\rref{system} (see also \cite{ABJKKMPR}). Note that in
the neutrally stable case $K$ does not 
depend on 
$T$ and $\mu$. 
We next turned to the case where $A$ is
not stable. In such a situation, even in the one-dimensional case, a
stabilizer $K$ cannot be chosen independently of $T$ and $\mu$. In
\cite{MCSS}, we considered the first nontrivial unstable case,
namely the double integrator $\dot x=J_2x+\al b_0u$, where $J_2$ denotes the
$2\times2$ Jordan block, the control is scalar and $b_0=(0,1)^T$. We showed that, for every
pair $(T,\mu)$, there exists a $(T,\mu)$-stabilizer for
$\dot x=J_2x+\al b_0u$, $\al\in \G(T,\mu)$.

In this paper, we restrict ourselves to the single-input case 
\begin{equation}\label{system-si}
\dot x = A x +\alpha(t)b u,~~~~~~~u\in\R,~~~~~~~~~\al\in{\cal G}(T,\mu),
\end{equation}
and 
we provide two sets of results. The first one
concerns the stabilizability of \rref{system-si}. 
Given two arbitrary constants $\mu$ and $T$ with $0<\mu\leq T$, we prove
the existence of a $(T,\mu)$-stabilizer for \r{system-si} when the eigenvalues of $A$
have non-positive real part. The second set of results concerns the
possibility of obtaining
an arbitrary rate of convergence once stabilization is
achieved. We essentially focus on the two-dimensional case and we
point out an interesting phenomenon: there exists $\rho_*\in (0,1)$
so that, for every controllable two-dimensional pair $(A,b)$,
every $T>0$ and every $\mu\in (0,\rho_* T)$, the maximal rate of
convergence of \r{system-si}
is finite.
Here maximality is evaluated with respect to all possible
$(T,\mu)$-stabilizers. 
As a consequence, we prove the existence of
matrices $A$ (e.g., $J_2+\lambda \Id_2$ with $\lambda$ large enough)
such that for every $T>0$ 
and every $\mu\in (0,\rho_* T)$, the PE
system \r{system-si}
does not admit
$(T,\mu)$-stabilizers. 
The latter result is
rather surprising when one compares it with the following two facts:
let $\rho\in (0,1]$; $(i)$  given a sequence $(\al_{n})_{n\in\N}$ with $\al_{n}\in
\caG(T_n,\rho T_n)$ and $\lim_{n\rightarrow +\infty}T_n=0$, all its
weak-$\star$ limit points $\al_\star$ take values in
$[\rho,1]$ (see Lemma~\ref{propri1}) and $(ii)$ the two-dimensional switched system
$\dot x=J_2x+\al_\star b_0u$ can be stabilized, uniformly with
respect to $\al_\star\in L^{\infty}(\Rp,[\rho,1])$, with an
arbitrary rate of convergence. The weak-$\star$ convergence
considered in $(i)$ is the natural one in this context since it
renders  the input-output mapping continuous. 

Let us briefly
comment on the technics
used in this paper. 
First of all, it is clear that the notion of
{\it common Lyapunov function}, rather powerful in the realm of
switched systems, cannot be of (direct) help here since, at the
differential level, one can evolve with an unstable dynamics $\dot
x= Ax$, when $\alpha=0$ takes the value zero. 
More refined tools as multiple 
and non-monotone Lyapunov functions (see, e.g., \cite{aeyels-peutemanTAC,aeyels-peutemanAUTOMATICA,branicky,colaneri,peutemanaeyels,shorten-review}) do not seem well-adapted to persistently excited systems, at least for what concerns the proof of their stability. 
It seems to us that one must rather perform
a trajectory analysis, on a time interval of length at least 
equal to $T$, in order to achieve any information which is
uniform with respect to $\alpha\in\G(T,\mu)$. This viewpoint is more similar to the geometric approach to switched systems behind the results in \cite{boscain-balde,boscainSIAM,selfC}.  
As a  second consideration, notice that point $(i)$
described above, which is systematically used in the paper,
presents formal similarities with the technique of {\it averaging} but
is rather different from it, since 
no periodicity nor constant-average assumption is made here.  
Moreover, for a given persistently excited system, $T$ is fixed and thus it does not tend to zero.

The paper is organized as follows. In Section~\ref{def00} we
introduce the notations of the paper, the basic definitions and some
useful technical lemmas. We gather in Section~\ref{s_n-integrator}
the stabilizability results for matrices whose spectrum has
non-positive real part. Finally, the analysis of the maximal rates
of convergence and divergence is the object of
Section~\ref{s_maximal-rate}. Since many of our results give rise to
further challenging questions, we propose in Section~\ref{OP}
several conjectures and open problems.

\section{Notations and definitions}\label{def00}
Let $\N$ denote the set of positive integers. Given $n$ and $m$
belonging to $\N$, we use $0_{n\times m}$ to denote the $n\times m$
matrix made of zeroes, $M_n(\R)$ the set of real-valued $n\times n$
matrices, and $\Id_n$ the $n\times n$ identity matrix. We also write
$0_n$ for $0_{n\times 1}$, $\sigma(A)$ for the spectrum of a matrix
$A\in M_n(\R)$, and $\Re(\lb)$ (respectively, $\Im(\lb)$) for the
real (respectively, imaginary) part of a complex number $\lambda$.

\begin{deff}[PE signal and $(T,\mu)$-signal]\label{Tm-signal}
Let $\mu$ and $T$ be positive constants with $\mu\leq T$. A \emph{$(T,\mu)$-signal} is
a measurable function $\alpha:\Rp\to[0,1]$
satisfying
\begin{equation}\label{EP}
\int_t^{t+T}\alpha(s)ds \geq \mu\,,\quad \forall t\in\mRp\,.
\end{equation}
We use ${\cal {G}}(T,\mu)$ to denote the set of all
$(T,\mu)$-signals. A \emph{PE signal} is a measurable function
$\alpha:\Rp\to[0,1]$ such that there exist $T,\mu$
for which $\alpha$ is a $(T,\mu)$-signal.
\end{deff}

\begin{deff}[PE system]\label{Tm-sys}
Given two positive constants $\mu$ and $T$ with $\mu\leq T$ and
a controllable pair $(A,b)\in M_n(\R)\times \R^n$, we define the
 \emph{PE system} associated to $T,\mu,A$, and $b$ as the family of linear control
systems given by \be\label{system1} \dot x=Ax+\alpha u b, \ \ \
\alpha\in {\cal {G}}(T,\mu). \ee
\end{deff}
Given a PE system \rref{system1}, we address the following problem.
We want to stabilize \rref{system1} \emph{uniformly} with respect to
every $(T,\mu)$-signal $\alpha$, i.e., we want to find a vector
$K\in\R^n$ which makes the origin of
\begin{equation}\label{feedback}
\dot x=(A-\alpha(t) b \T{K}) x
\end{equation}
globally asymptotically stable, with $K$ depending only on $A$, $b$,
$T$ and $\mu$.

More precisely, referring to
$
x(\cdot\,;t_0,x_0,K,\alpha)
$ 
as the solution of \rref{feedback} with initial condition
$x(t_0;t_0,x_0,K,\alpha)=x_0$, we introduce the following
definition.

\begin{deff}[$(T,\mu)$-stabilizer]\label{stab}
Let $\mu$ and $T$ be positive constants with $\mu\leq T$. The gain $K$ is said to be a
\emph{$(T,\mu)$-stabilizer} for \rref{system1} if \rref{feedback}
 is globally asymptotically stable, uniformly with every $(T,\mu)$-signal
 $\alpha$. Since \rref{feedback} is linear in $x$, this is equivalent to
 say that \rref{feedback}
is exponentially stable, uniformly with respect to $\alpha\in
\mathcal G(T,\mu)$, i.e., there exist $C,\gamma>0$ such that every
solution $x(\cdot\,;t_0,x_0,K,\alpha)$ of \r{feedback} satisfies
$$\|x(t;t_0,x_0,K,\alpha)\|\leq C e^{-(t-t_0)\gamma}\|x_0\|,\ \ \
\forall t\geq t_0.$$
\end{deff}

The next two lemmas collect some properties of PE
signals.

\begin{lemma}\label{proprieta}
\begin{enumerate}
\item \label{shift-i}If $\alpha(\cdot)$ is a $(T,\mu)$-signal, then,
for every $t_0\geq
0$, the same is true for $\alpha(t_0+\cdot)$. 
\item \label{rho202} If $0<\rho'<\rho$ and $T>0$ then
$\G(T,\rho T)\subset \G(T,\rho' T)$.
\item \label{ligio}
For $\eta\in (0,\mu)$, $\G(T,\mu)\subset \G(T+\eta, \mu)\cap
\G(T-\eta, \mu-\eta)$.
\item \label{rho2} If $T\geq \tau>0$ and $\rho>0$, then
$\G(\tau,\rho \tau)\subset \G(T,(\rho/2) T)$.
\item \label{rhon} For every $0<\rho'<\rho$
there exists $M>0$
such that for every $T\geq M\tau>0$ one has 
$\G(\tau,\rho\tau)\subset \G(T,\rho'T)$.
\end{enumerate}
\end{lemma}
\begin{proof}
We only provide an argument for points \ref{rho2} and \ref{rhon}.
Fix $t\geq 0$, $T\geq \tau$, $\rho>0$ and $\al\in\G(\tau,\rho\tau)$. Let
$l$ be the integer part of $T/\tau$.
Since $l\geq \max(1,T/\tau-1)$,  then $\int_t^{t+T}\al(s)ds\geq l\rho \tau\geq
\max(\tau,T-\tau)\rho\geq T\rho/2$.
For $\rho'\in (0,\rho)$ and $T/\tau$ large enough, one has $\max(\tau,T-\tau)\geq (\rho'/\rho)T$
and so $\int_t^{t+T}\al(s)ds\geq  \rho' T$.
\end{proof}

 Let
 $$b_0=\lp \ba{c}0\\ 1\ea\rp,\ \ \ \ A_0=\lp\ba{cc}0&1\\
-1&0\ea\rp.
$$

Recall that an element $f$ of $L^\infty(\Rp,[0,1])$ is the
weak-$\star$ limit of a sequence $(f_k)_{k\in\N}$ of elements of
$L^\infty(\Rp,[0,1])$ if, for every $g\in L^1(\Rp,\R)$,
\begin{equation}\label{weakstar}
\int_0^\infty f(s)g(s)ds=\lim_{k\rightarrow\infty}\int_{0}^\infty f_k(s)g(s)ds.
\end{equation}
It is well known that $L^\infty(\Rp,[0,1])$ endowed with the
weak-$\star$ topology is compact (see, for instance,
\cite{BRE}). Hence, each $\G(T,\mu)$ is weak-$\star$ compact.
Unless specified,
limit points of sequences of PE signals are to be understood
as limits of subsequences
with respect to the weak-$\star$ topology of
$L^\infty(\Rp,[0,1])$.

\begin{lemma}\label{propri1}
Let $(\alpha^{(n)})_{n\in\N}$ and
$(\nu_n)_{n\in\N}$ be, respectively,  a sequence of $(T,\mu)$-signals
and an increasing sequence of positive real
numbers such that $\lim_{n\rightarrow \infty}\nu_n=\infty$.
\begin{enumerate}
\item \label{P1-1}
Define $\alpha_n$ as the $(T/\nu_n,\mu/\nu_n)$-signal given by
$\alpha_n(t)=\alpha^{(n)}(\nu_n t)$ for $t\geq 0$. If $\alpha_\star$
is a limit point of the sequence $(\alpha_n)_{n\in\N}$, then
$\alpha_\star$ takes values in $[\mu/T,1]$ almost everywhere.
\item \label{P1-2} \label{per}
Let $j_0\in \{0,1\}$ and $h\in\N$. Let $\om_j$, $j=j_0,\dots,h$, be
real numbers with $\om_j=0$ if and only if $j=0$ and $\{\pm
\om_j\}\ne \{\pm \om_l\}$ for $j\ne l$. For every $t\geq 0$, let
$$
v(t)= \lp\ba{c}1\\ e^{\om_1 A_0 t}b_0\\ \vdots \\ e^{\om_h A_0
t}b_0\ea\rp\ \mbox{if }j_0=0\ \ \mbox{ or }\ \ v(t)=
\lp\ba{c}e^{\om_1 A_0 t}b_0\\ \vdots \\ e^{\om_h A_0 t}b_0\ea\rp\
\mbox{if }j_0=1.
$$

For every signal $\al$ and every $t\geq 0$, define
\begin{equation}\label{peC}
\alpha^C(t)=\alpha(t)v(t)v(t)^T.
\end{equation}
Then $\alpha^C$ is a time-dependent non-negative symmetric $(2h+1-j_0)\times (2h+1-j_0)$ matrix with
$\alpha^C\leq \Id_{2h+1-j_0}$ and there exists $\xi>0$ only depending on
$T,\mu$ and $\omega_{j_0},\dots,\om_h$ such that, for every $t\geq 0$, 
\begin{equation}\label{mPE}
\int_t^{t+T}\alpha^C(\tau)d\tau\geq \xi\, \Id_{2h+1-j_0}.
\end{equation}

Define, moreover,
$\alpha^C_n(t)=(\alpha^{(n)})^C(\nu_n t)$ for every $t\geq 0$ and every $n\in \N$. If $\alpha^C_\star$ is
a limit point  of the sequence $(\alpha^C_n)_{n\in\N}$ for the weak-$\star$ topology of
$L^\infty(\Rp,M_{2h+1-j_0}(\R))$,
then $\alpha_\star^C\geq (\xi/ T) \Id_{2h+1-j_0}$ almost everywhere.
\end{enumerate}
\end{lemma}
{\it Proof.}
Let us first prove point~\ref{P1-1}. Let $\alpha_\star$ be the
weak-$\star$ limit of some sequence $(\alpha_{n_k})_{k\geq 1}$.
For every interval $J$ contained in $\mRp$ of finite length $|J|>0$, apply
\r{weakstar} by taking as $g$  the characteristic function of $J$.
Since each
$\alpha_{n_k}$ is a $(T/{\nu_{n_k}},\mu/{\nu_{n_k}})$-signal, it
follows that
\[
\frac 1{|J|} \int_J \alpha_\star(s)ds=\lim_{k\to\infty} \frac 1{|J|}
\int_J \alpha_{n_k}(s)ds \geq \liminf_{k\to\infty}
\frac{\mu}{|J|\nu_{n_k}} {\cal
I}\left(\frac{|J|\nu_{n_k}}{T}\right)=\frac{\mu}{T}\,,
\]
where ${\cal I}(\cdot)$ denotes the integer part.
Since $\al_\star$
is measurable and bounded, 
almost
every $t>0$ is a Lebesgue point for $\al_\star$, {\em i.e.}, the
limit
\[\lim_{\eps\to0+}\frac 1 {2\eps} \int_{t-\eps}^{t+\eps} \alpha_\star(s)ds\]
exists and is equal to $\al_\star(t)$ (see, for instance,
\cite{rudin}). We conclude that, as claimed,
$\alpha_\star\geq \mu/T$ almost everywhere.

For the first part of point~\ref{P1-2}
fix
$t\geq 0$ and notice that the map
$$\al\mapsto\int_t^{t+T}\alpha^C(s)ds$$
is continuous with
respect to the weak-$\star$ topology and takes values in the set of
non-negative symmetric matrices.

We claim that all such matrices are positive definite. 
Assume  by
contradiction that there exist $\alpha\in \mathcal G(T,\mu)$ and
$x_0\in\R^{2h+1-j_0}\setminus\{0_{2h+1-j_0}\}$ such that $\int_t^{t+T}x_0^T\alpha^C(s)x_0ds=0$.
Then, for almost every $s\in [t,t+T]$, we would have
$\alpha(s)x_0^T v(s)=0$. Since $\alpha(s)\neq 0$ for $s$ in a
set of positive measure, we deduce that the real-analytic function
$x_0^T v(\cdot)$ 
is identically equal to zero.
Let $A_0^C=\diag(1,\om_1 A_0,\dots,\om_h A_0)$ if $j_0=0$ or
$A_0^C=\diag(\om_1 A_0,\dots,\om_h A_0)$ if $j_0=1$.
Then $x_0^T (A_0^C)^jv(0)=0$ for every non-negative integer $j$.
The contradiction is reached, since $(A_0^C,v(0))$ is a controllable pair
and $x_0\ne 0_{2h+1-j_0}$.

Then, by weak-$\star$ compactness of $\G(T,\mu)$, we deduce the
existence of $\xi>0$ independent of $\al$ such that \r{mPE} holds true.
The independence of $\xi$ with respect to $t$ follows from the
shift-invariance of $\G(T,\mu)$ pointed out in
Lemma~\ref{proprieta}. 

The second part of point~\ref{P1-2} follows from  the same argument used to prove point~\ref{P1-1}, noticing that, for
every $t\geq 0$,  
$$
\int_t^{t+\frac T{\nu_n}}\alpha^C_n(\tau)d\tau\geq \frac{\xi}{\nu_n}
\Id_{2h+1-j_0}.\eqno{\EOP}
$$

\section{Spectra with non-positive real part}\label{s_n-integrator}

We consider below the problem of whether a controllable pair
$(A,b)$ gives rise to a PE system that can be $(T,\mu)$-stabilized for
every choice of $\mu$ and $T$. We will see in Section~\ref{s_maximal-rate}
that this cannot
be done
in general. 
The scope of this section is
to study the case in which each eigenvalue of $A$ has non-positive real part.

The first step is to consider the special case of the $n$-integrator.
Let $J_n\in M_n(\R)$ be defined as
$$
J_n=
\lp\ba{cccccc}
0&1&0&\cdots&\cdots&0\\
0&0&1&0&\cdots&0\\
&&&&&\\
\vdots&&&\hspace{-7mm}\ddots&\hspace{-7mm}\ddots&\vdots\\
&&&&&\\
0&&\cdots&&0&1\\
0&&\cdots&\cdots&&0
\ea
\rp.
$$

\begin{thm}\label{theo}
Let $A=J_n$ and $b=(0,\dots,0,1)^T\in \R^n$.
Then, for every $T,\mu$ with $T\geq\mu>0$
there exists a $(T,\mu)$-stabilizer for \r{system1}.
\end{thm}
\begin{proof}
In the special case of the  $n$-integrator
system \r{feedback} becomes
\be\label{NIf}
\left\{ \ba{lcl}
\dot x_j&=&x_{j+1},\ \ \mbox{for $j=1,\dots,n-1$},\\
\dot x_n&=&-\alpha(t) (k_1 x_1+\cdots+k_n x_n)\,, \ea
\right.
\ee
where  $K=\T{(k_1,\dots,k_n)}$.

 For every $\lbb>0$, define 
 \be\label{Dnv}
 D_{n,\lbb}=\diag(\lbb^{n-1},\dots,\lbb, 1).
 \ee
As done in \cite{MCSS} in
the case $n=2$, one easily checks that, in accordance with
\be\label{elementary}
\lbb D_{n,\lbb}^{-1} J_n D_{n,\lbb}=J_n,\ \ \ D_{n,\lbb}
b=b,
\ee
the time-space
transformation
\be\label{tst0} x_\lbb(t)=D_{n,\lbb}^{-1} x(\lbb
t)\,,\quad \forall t\geq \frac {t_0}\nu,
\ee
of the trajectory
$x(\cdot)=x(\cdot\,;t_0,x_0,K,\alpha)$
satisfies
\brs \frac{d}{dt}x_\lbb(t)&=&J_n x_\lbb(t)-\al(\lbb t)\lbb b
K^T D_{n,\lbb} x_\lbb(t),
\ers
that is,
\be\label{Nro}
x_\lbb(\cdot)=x(\cdot\,;t_0/\nu,
D_{n,\lbb}^{-1}
x_0,\lbb D_{n,\lbb} K,\alpha(\lbb\, \cdot)).
\ee
As a consequence, \r{NIf} admits a
$(T,\mu)$-stabilizer if and only if it admits a
$(T/\lbb,\mu/\lbb)$-stabilizer.
More precisely,
$K$ is a
$(T,\mu)$-stabilizer if and only if $\lbb D_{n,\lbb} K$ is a
$(T/\lbb, \mu/\lbb)$-stabilizer.


Let us introduce, for every gain $K$,  the 
switched system 
\be\label{sw_GK}
 \left\{ \ba{lcl}
\dot x_j&=&x_{j+1},\ \ \mbox{for $j=1,\dots,n-1$},\\
\dot x_n&=&-\alpha_\star(t) (k_1 x_1+\cdots+k_n x_n), \ea
\right.\ \ \ \ \ \ \ \ \ \ \ \al_\star\in L^\infty(\Rp,[\mu/T,1]).
\ee

Recall that 
\r{sw_GK} is said to be {\it globally uniformly exponentially stable} as a switched system 
if 
the origin is globally exponentially stable, uniformly with respect to 
$\al_\star\in L^\infty(\Rp,[\mu/T,1])$,  for the dynamics of \r{sw_GK}. (For this and other notions of stability of switched systems see, for instance, \cite{liberbook}.) 
For every $K$ such that 
$k_1\ne 0$, define $X_1=k_1 x_1+\cdots+k_n x_n$, $X_2=k_1
x_2+\cdots+k_{n-1} x_n$, \dots, $X_n=k_1 x_n$. 
The global uniform exponential stability of  \r{sw_GK} is clearly equivalent to
that of 
\be\label{sw_GK_X} \dot X_j=X_{j+1}-\al_\star \bar k_j X_1,\
\ \ \ \ \ j=1,\dots,n,\ \ \al_\star(t)\in [\mu/T,1], \ee 
where $\bar k_{j}=k_{n+1-j}$ and, by convention, $X_{n+1}=0_{n}$. 

It has been proven in Gauthier and Kupka
\cite[Lemma 4.0]{GK} (where the result is attributed to W.P. Dayawansa), 
that there exist 
$\overline{K}\in\R^n$, a scalar $\gamma>0$ and a symmetric positive
definite $n\times n$ matrix $S$ such that 
\be\label{GK} 
\T{(J_n-\bar\alpha {\overline{K}}(1,0,\dots,0))}S+S(J_n-\bar \alpha
{\overline{K}}(1,0,\dots,0))\leq -\gamma \Id_n, 
\ee 
for every
(constant) $\bar \alpha\in [\mu/T,1]$. 

Hence, there exist 
a gain $K\in \R^n$ such that \r{sw_GK} 
is globally uniformly exponentially stable  
and a positive definite
matrix $S'$ such that the 
quadratic Lyapunov
function
 $V(x)=\T{x}S' x$ decreases
uniformly on all trajectories of \r{sw_GK}.
In particular, there exists a time $\tau$ such that
every trajectory of \r{sw_GK} starting in
$B^V_2=\{x\in \R^n\mid V(x)\leq 2\}$
at time $0$ lies in
$B^V_1=\{x\in \R^n\mid V(x)\leq 1\}$
for every time larger than $\tau$.

We claim that, for some $\lbb>0$, every trajectory 
of $\dot x=(A-\alpha_\lbb(t) b \T{K}) x$
with initial condition in $B_2^V$ and 
corresponding to  a
$(T/\lbb,\mu/\lbb)$-signal $\al_\lbb$  stays in $B_1^V$ for every time larger than 
$2\tau$. (In particular, by homogeneity, $K$ is a $(T/\lbb,\mu/\lbb)$-stabilizer
and thus $\lbb^{-1}D_{n,\lbb}^{-1}K$ is a $(T,\mu)$-stabilizer.)
 Assume, by contradiction, that for every $l\in\N$ 
there exist 
$x_{0,l}\in B_2^V$, $t_l\in[2\tau,4\tau]$ and $\al_l\in \G(T/l,\mu/l)$ such that 
\be\label{da_contraddire}
x(t_l;0,x_{0,l},K,\al_l)\not\in B_1^V\ \
\mbox{ for every $l\in\N$}.
\ee
By compactness of $B_2^V\times [2\tau,4\tau]$ and by weak-$\star$ compactness of
$L^\infty(\Rp,[0,1])$, we can assume that, up to extracting a
subsequence, $x_{0,l}\to x_{0,\star}\in B_2^V$, $t_l\to t_\star\in
[2\tau,4\tau]$ and $\al_l$ converges weakly-$\star$ to
$\alpha_\star\in L^\infty(\Rp,[0,1])$ as $l$ goes to infinity. Then
$x(t_l;0,x_{0,l},K,\al_l)$ converges, as $l$ goes to infinity, to
$x(t_\star;0,x_\star,K,\al_\star)$ (see \cite[Appendix]{MCSS} for
details). Since $\al_\star\geq \mu/T$ almost everywhere
(point~\ref{P1-1} of Lemma~\ref{propri1}), then $\al_\star$ can
be taken as an admissible signal in \r{sw_GK}.

By homogeneity of the linear system \r{sw_GK}
and because $t_\star\geq 2\tau$, we have that
$$V(x(t_\star;0,x_\star,K,\al_\star))\leq 1/2.$$
Therefore, for $l$
large enough $x(t_l;0,x_{0,l},K,\al_l)\in B_1^V$ contradicting
\r{da_contraddire}.
\end{proof}

Let us now turn the general case where the spectrum of $A$ has non-positive real part. The main technical difficulties
in order to adapt the proof of Theorem~\ref{theo} 
come from the fact that $A$ may have several Jordan blocks of different sizes.

\begin{thm}\label{thm2}
Let $(A,b)\in M_n(\R)\times\R^n$ be a controllable pair and assume
that the eigenvalues of $A$ have non-positive real part. Then, for
every $T,\mu$ with $T\geq\mu>0$ there exists a $(T,\mu)$-stabilizer for
\r{system1}.
\end{thm}
\begin{proof}
Fix a controllable pair $(A,b)\in M_n(\R)\times \R^n$.
Up to a linear change of
variable, $A$ and $b$ can be written as
$$
A=\begin{pmatrix}
A_1&A_3\\
0_{(n-n')\times n'}&A_2
\end{pmatrix}, \ \ \
b=\begin{pmatrix} b_1\\b_2
\end{pmatrix},
$$
where $n'\in\{0,\dots,n\}$, $A_1\in M_{n'}(\R)$ is Hurwitz and all the eigenvalues of $A_2\in M_{n-n'}(\R)$ have zero
real part.  From the
controllability assumption, we deduce that $(A_2,b_2)$ is
controllable. Setting $x=(x_1^T,x_2^T)^T$ according to the above decomposition,
system \rref{system} can be written as
\br
\dot x_1&=&A_1x_1+A_3x_2+\alpha(t)b_1u,\label{h1}\\
\dot x_2&=&A_2x_2+\alpha(t)b_2u.\label{h2}
\er
If 
there
exists a $(T,\mu)$-stabilizer $K_2$  for
 \r{h2}, then 
 $$K=\lp\ba{c}0_{n'}\\ K_2\ea\rp$$
is a $(T,\mu)$ stabilizer for \r{system}. 
It is therefore enough to prove the theorem under the extra hypothesis 
 that all eigenvalues of
$A$ lie on the imaginary axis.

Denote the distinct eigenvalues of $A$ by $\pm i\omega_j$, $j\in\{j_0,j_0+1,\dots,h\}$,
where $j_0=1$ if $0\not\in\sigma(A)$ and
$j_0=0$ with $\omega_0=0$ otherwise.
For every $j\in\{0,\dots,h\}$, let $r_j$ be the
multiplicity of
$i \omega_j$, with the convention that $r_0=0$ if $0\not\in\sigma(A)$.

Assume that $A$ is decomposed in Jordan blocks.
Since $(A,b)$ is controllable,
then
$A$ has
a unique (complex) Jordan block associated with each $\{i\omega_j,
-i\omega_j\}$, $j_0\leq j\leq h$. (Otherwise,
the rank of the matrix
$(A-i\omega_j \Id_n \mid  b)$ would be strictly smaller than $n$,
contradicting the Hautus test for controllability.)
Therefore, for every $j=1,\dots,h$, the Jordan block associated to $i\om_j$ is
$\omega_j A^{(j)}+J^{C}_{r_j}$,
where $A^{(j)}=\diag(A_0,\dots,A_0)\in M_{2r_j}(\R)$ and $J^C_{r_j}\in M_{2r_j}(\R)$
is defined as
$$
J^C_{r_j}= \lp\ba{cccccc}
0_{2\times 2}&\Id_2&0_{2\times 2}&\cdots&\cdots&0_{2\times 2}\\
0_{2\times 2}&0_{2\times 2}&\Id_2&0_{2\times 2}&\cdots&0_{2\times 2}\\
&&&&&\\
\vdots&&\hspace{-7mm}\ddots&\hspace{-7mm}\ddots&\hspace{-5mm}\ddots\hspace{5mm}
\ddots&\vdots\\
&&&&&0_{2\times 2}\\
0_{2\times 2}&&\cdots&0_{2\times 2}&0_{2\times 2}&\Id_2\\
0_{2\times 2}&&\cdots&\cdots&0_{2\times 2}&0_{2\times 2} \ea\rp,
$$
that is, in terms of the Kronecker product, $J_{r_j}^C=J_{r_j}\otimes \Id_2$.

All controllable
linear control systems associated with a pair
$(A,b)$ that have in common
the eigenvalues of $A$, counted according to their multiplicity,
 are state-equivalent, since
they can be transformed by a
 linear transformation of coordinates
 into the same system
 under companion
 form.
We exploit such an equivalence to deduce that, up to a linear
transformation of coordinates, \r{system} can be written as
\be\label{NIcP0} \left\{ \ba{lcl} \dot x_0&=&J_{r_0}x_0+\alpha
b^{0}u,\\
\dot x_j&=&(\omega_j A^{(j)}+J^{C}_{r_j})x_j + \alpha b^{j}u,\ \
\mbox{for $j=1,\dots,h$},  \ea \right.
\ee
where
$b^{0}$ and
$b^{j}$ are respectively the vectors of $\R^{r_0}$  and $\R^{2r_j}$
with all coordinates equal to zero except the last one that is equal to one.
Here $x_0\in \R^{r_0}$ and $x_j\in \R^{2r_j}$ for $j=1,\dots,h$

Write the feedback law as
$u=-K^Tx=-K_0^Tx_0-\sum_{l=1}^h K_l^Tx_l$
with
$K_0\in \R^{r_0}$ and $K_j\in\R^{2r_j}$
for every $1\leq j\leq h$.

For every $\nu>0$ consider the following change of time-space variables:
let
\brs
y_0(t)&=& D_{r_0,\nu}^{-1} x_0(\nu t),\\
y_j(t)&=&(D_{r_j,\lbb}^C)^{-1} e^{-\nu t A^{(j)}}x_j(\nu t), \ \ \ \hbox{for }
1\leq j\leq h,
\ers
 where $D_{r_0,\lbb}$ is defined as in \r{Dnv} and 
 $$D_{r_j,\lbb}^C=D_{r_j,\lbb} \otimes \Id_2\in M_{2r_j}(\R).$$

In accordance with
$$
\lbb (D_{r_j,\lbb}^C)^{-1} J^C_{r_j} D_{r_j,\lbb}^C=J^C_{r_j},\ \ \ D_{r_j,\lbb}^C b^j=b^j,
$$
we end
up with the following linear time-varying system
\be\label{NIcP1}
\left\{ \ba{lcl}
\dot y_0&=&J_{r_0}y_0-\alpha_\nu(t)
b^{0}\big(K_{0,\nu}^Ty_0+\sum_{l=1}^h
K_{l,\nu}^T e^{\nu t \omega_l A^{(l)}}  y_l\big),\\
\dot y_j&=&J^{C}_{r_j}y_j - \alpha_\nu(t)
b^{j,\nu}(t)\big(K_{0,\nu}^T y_0+\sum_{l=1}^h K_{l,\nu}^T e^{\nu t \omega_l A^{(l)}} y_l\big),
\ \ \mbox{for $j=1,\dots,h$},
\ea \right.
\ee
where $K_{0,\nu}=\nu
D_{r_0,\nu} K_0$, $K_{j,\nu}=\nu D_{r_j,\nu}^C K_j$
and $b^{j,\nu}(t)=e^{-\nu t \omega_j A^{(j)}}b^{j}$ for
$j=1,\dots,h$.
Given $\nu>0$, \rref{feedback} admits a $(T,\mu)$-stabilizer if and only if \rref{NIcP1}
admits a $(T/\nu,\mu/\nu)$-stabilizer.



For each $l=1,\dots,h$, assume that $K_l^T$ is of the form
$(0,k_1^l,\dots,0,k_{r_l}^l)$, that is,
$$ K_l^T=\K_l\otimes (0,1),\ \ \K_l=(k_1^l,\dots,k_{r_l}^l).$$
For uniformity of notations, we also write
$\K_0=K_0^T$.

Let $(\al_\nu)_{\nu>0}$ be a family of signals satisfying
$\al_\nu\in \G(T/\nu,\mu/\nu)$ for every $\nu>0$.
Consider a sequence $(\nu_n)_{\n\in\N}$ going to infinity as $n\to \infty$ such
that the matrix-valued curve $\al_{\nu_n}^C(\cdot)$, defined as in \r{peC}, has a weak-$\star$ limit as $n\to\infty$
 in $L^\infty(\Rp,M_{2h+1-j_0}(\R))$.
Denote the weak-$\star$ limit by $C_\star$.
It follows form point~\ref{per} of Lemma~\ref{propri1} that
$C_\star(t)$ is symmetric and
$$
C_\star(t)\geq \xi \Id_{2h+1-j_0},
$$
for almost every $t\geq 0$,
for some positive scalar $\xi$ only depending on $T,\mu$ and $\sigma(A)$.

Define the $2\times 2$ time-dependent matrices $C_{jl}$, $1\leq j,l\leq h$,
the $1\times 2$ time-dependent matrices $C_{0j}$, $1\leq j\leq h$, and the
scalar time-dependent signal $C_{00}$ by the relation
$$
C_\star=(C_{jl})_{j_0\leq j,l\leq h}.
$$

Consider, for every $n\in \N$, system
\rref{NIcP1} with $\nu=\nu_n$ and $K_{\nu}=K$.
All coefficients of the sequence of systems obtained in this way  
 are
weakly-$\star$ convergent as $n$ goes to infinity. 
The limit system is 
\be\label{NIcP4}
\left\{ \ba{lcl} \dot
y_0&=&J_{r_0}y_0-b^{0}\big(
C_{00} {\K}_0 y_0+\sum_{l=1}^h
C_{0l}
(\K_{l}\otimes \Id_2)y_l\big),\\
\dot y_j&=&J^{C}_{r_j}y_j
-(b^j\otimes \Id_2)\left(C_{0j}^T \K_0 y_0+\sum_{l=1}^h
C_{jl}
(\K_{l}\otimes \Id_2) y_l\right), \ \ \mbox{for
$j=1,\dots,h$}. \ea \right.
\ee
We consider \r{NIcP4} as a switched system depending on $K$ 
in which the admissible switching laws are all the time-varying matrix-valued coefficients $C_{jl}$ obtained from the limit procedure described above.

In the sequel, we only treat the case where $0$ is not an eigenvalue
of $A$. The general case presents no extra mathematical difficulties
and can be treated similarly.
Then system \rref{NIcP4}
takes the form
\be\label{NIcP44}
\dot y_j=J^{C}_{r_j}y_j
-(b^j\otimes \Id_2)\sum_{l=1}^h
C_{jl}
({\K}_{l}\otimes \Id_2) y_l, \ \ \mbox{for
$j=1,\dots,h$}.
\ee
We also assume that the multiplicities
$r_1,\dots,r_h$
of the eigenvalues of $A$ form a non-increasing sequence.

Let us impose a further restriction on the structure of the feedback $
{K}$.
Assume that there exist $\bar k_1,\dots,\bar k_{r_1}\in\R$, each of them
different from zero,
such that
$$k_\xi^l=\bar k_{r_l+1-\xi},\ \ \ \ \hbox{for }1\leq l\leq h\hbox{ and }1\leq \xi\leq r_l.
$$

We find it  useful to provide an equivalent representation of
system~\rref{NIcP44} in a higher dimensional vector space,
introducing some redundant variables. In order to do so, for $l\in\{1,\dots,r_1\}$, associate to $y=(y_1,\dots,y_h)$ the $2h$-vector
$$Y_l=\left(\ba{c}
(\K_1\otimes \Id_2)(J^{C}_{r_1})^{l-1}y_1\\
\vdots\\
(\K_h\otimes \Id_2)(J^{C}_{r_h})^{l-1}y_h
\ea\right).$$
Notice
that 
the last $2h-2m_l$ coordinates of $Y_l$ are
equal to zero, where
$m_l$ denotes the number of Jordan blocks of $A$
of size not smaller than $l$, that is,
$$m_l=\#\{j\mid 1\leq j\leq h,\ r_j\geq l\}.$$
For $l\in\{1,\dots,r_1\}$, let
$p_l$ be the orthogonal projection of $\R^{2 h}$ onto $\R^{2 m_l}\times\{0_{2h-2 m_l}\}$, i.e.,
$$p_l=\diag(\Id_{2 m_l},0_{(2h-2m_l)\times(2h-2m_l)}).$$
By construction we have
$p_1=\Id_{2r_1}$ and $p_lY_j=Y_j$ for $1\leq l\leq j\leq r_1$.

Notice that the map $(y_1,\dots,y_h)\mt (Y_1,\dots,Y_{r_1})$ is a bijection
between $\R^n$ and the subspace $E^h_{m_1,\dots,m_{r_1}}$ of $\R^{2h r_1}$ defined by
$$E^h_{m_1,\dots,m_{r_1}}=\{(Y_1,\dots,Y_{r_1})\mid Y_l\in \R^{2 h}\mbox{ and }p_l Y_l=Y_l\mbox{ for }l=1,\dots,r_1\}.$$
Indeed, the matrix corresponding to the transformation is upper triangular, with
the $\bar k_l$'s as elements of the diagonal, if one considers the
following choice of coordinates on $E^h_{m_1,\dots,m_{r_1}}$: 
take the first two coordinates of the first copy of $\R^{2 h}$, then the first two of its second copy and so on until the $r_1^\mathrm{th}$ copy; then take the third and fourth coordinates of the first copy of $\R^{2 h}$ and repeat the procedure until its $r_2^\mathrm{th}$ copy; and so on, until the last two coordinates of the $r_h^\mathrm{th}$ copy of $\R^{2 h}$.

If $y$ is a solution of system~\rref{NIcP44}, then
$Y=(Y_1,\dots,Y_{r_1})$ is a trajectory in $E^h_{m_1,\dots,m_{r_1}}$
satisfying the system of equations
\be\label{NIcP5}
\dot{Y_l}=Y_{l+1}
-\bar k_l p_l C_\star Y_1, \ \ \mbox{for
$l=1,\dots,r_1$},
\ee
where, by convention, $Y_{r_1+1}=0_{2h}$.

We 
prove in the following proposition
that there exist $\bar k_1,\dots,\bar
k_{r_1}\ne 0$ such that system~\r{NIcP5}, restricted to
$E^h_{m_1,\dots,m_{r_1}}$, is exponentially stable uniformly with
respect to all time-dependent measurable 
 symmetric matrices
 $C_\star$  satisfying
 $\xi \Id_{2h} \leq C_\star(t)\leq \Id_{2h}$ almost everywhere.


%

\begin{prop}\label{ult0}
For every $h,r_1\in\N$, for every non-increasing sequence of
non-negative numbers
$m_1,\dots,m_{r_1}$ such that $m_1\leq h$ and for every
 $\xi>0$, there exist $\lambda,\bar k_1,\dots,\bar k_{r_1}>0$ and
  a symmetric positive definite $2hr_1\times 2hr_1$ matrix
  $S$ such that, for
  every $C_\star\in L^\infty(\Rp,M_{2h}(\R))$, if $C_\star(t)$ is
  symmetric and satisfies $\xi \Id_{2h} \leq C_\star(t)\leq \Id_{2h}$
  almost everywhere, then
  any solution
$Y: \mRp\rightarrow E^h_{m_1,\dots,m_{r_1}}$
of \rref{NIcP5}
satisfies for almost every $t\geq 0$ the inequality
$$
\frac d{dt}\lp Y(t)^T S Y(t)\rp\leq -\lambda \|Y(t)\|^2.
$$
\end{prop}
\begin{proof}
The proof is similar to that of \cite[Lemma 4.0]{GK} and goes by induction
on $r_1$.

We start the argument for $r_1=1$, with $h\in \N$, $0\leq m_1\leq h$ and
$\xi>0$ arbitrary. In that case the
system reduces to
$$
\dot{Y_1}=-\bar k_1 p_1 C_\star Y_1,
$$
with $Y_1\in E^h_{m_1}=\R^{2 m_1}\times \{0_{2h-2 m_1}\}$. The conclusion
follows by taking $\bar k_1=1$ and $S=\Id_{2h}$.

Let $r_1$ be a positive integer. Assume that the proposition holds true for every positive
integer $j\leq r_1$ and for  every $h\in\N$, $0\leq m_1\leq\cdots \leq m_{r_1}\leq h$ and $\xi>0$.
Consider system \rref{NIcP5} where $l$ runs between $1$ and $r_1+1$.
Set $Y=(Y_2^T,\dots,Y_{r_1+1}^T)^T$.
Note that if $(Y_1^T,\dots,Y_{r_1+1}^T)^T\in E^h_{m_1,\dots,m_{r_1+1}}$, then
$Y\in E^h_{m_2,\dots,m_{r_1+1}}$.
The dynamics of  $(Y_1,Y)$ are given by
$$
\left\{ \ba{lcl}
\dot Y_1=-\bar k_1 C_\star Y_1+\Pi_1 Y,\\
\dot Y=-\overline{K} C_\star Y_1+
{\cal J}
Y,
\ea \right.
$$
where
\brs
\Pi_1&=&(\Id_{2h},0_{2h\times 2h(r_1-1)}),\\
\overline{K}&=&
\left(\ba{c} \bar k_2 p_2 \\ \vdots\\ \bar k_{r_1+1} p_{r_1+1}\ea\right),\\
{\cal J}
&=&J_{r_1}\otimes \Id_{2h}.
\ers

Define the linear
change of variables
$(Z_1,Z)$ given by
$$
Z_1=Y_1, \ \ \ \ Z=Y+\Omega
Y_1,
$$
where
$$\Omega=\left(\ba{c} \eta_2 p_{2}\\ \vdots \\ \eta_{r_1+1}p_{r_1+1}\ea\right)$$
and the $\eta_l$'s are scalar constants to be chosen later. Note
that $Z$ belongs to $E^h_{m_2,\dots,m_{r_1+1}}$ if $Y$ does.
The dynamics of  $(Z_1,Z)$ is given by
\be\label{NIcP7}
\left\{ \ba{lcl}
\dot Z_1=(-\bar k_1C_\star +\Pi_1 \Omega)Z_1+\Pi_1Z,\\
\dot Z=-\big((\overline{K}+\bar k_1\Omega) C_\star+ ({\cal J}
+\Omega\Pi_1)\Omega\big)Z_1+({\cal J}
+\Omega\Pi_1)Z. \ea \right. \ee Let us apply the induction
hypothesis to the system \beq\label{ult2}
\dot Z=({\cal J}
+\Omega\Pi_1)Z, \eeq which is well defined on
$E^h_{m_2,\dots,m_{r_1+1}}$ and has the same structure as
system~\r{NIcP5}. (Here $\C_\star\equiv \Id_{2h}$ and therefore one
can take as $\xi$ any positive constant smaller than one.)
We deduce the existence of $\lambda>0$, $\eta_l<0$, $2\leq l\leq
r_1+1$ and a symmetric positive definite matrix $S$ such that $\dot
V(t)\leq -\lambda \|Z(t)\|^2$ where $V(t)=Z(t)^TS Z(t)$ and $Z(t)$
is any trajectory of \rref{ult2} in $E^h_{m_2,\dots,m_{r_1+1}}$.
Therefore,
$$ \left.\left[({\cal J}
+\Om \Pi_1)^T S+ S({\cal J}
+\Om \Pi_1)\right]\right|_{E^h_{m_2,\dots,m_{r_1+1}}}\leq -\lb\,
\Id_{E^h_{m_2,\dots,m_{r_1+1}}}.$$

Since $\Om$ is fixed, for every $\bar k_1>0$ there exists a unique
$\overline{K}(\bar k_1)$ such that $\overline{K}(\bar k_1)+\bar
k_1\Om=0_{2 r_1 h\times 2 h}$. Assume that
$\overline{K}=\overline{K}(\bar k_1)$ and notice that the
corresponding  $\bar k_2,\dots,\bar k_{r_1+1}$ are positive. 

Choose $S'=(1/2)\diag(\Id_{2h},S)$ and define the corresponding Lyapunov function
$W(Z_1,Z)=\|Z_1\|^2/2+Z^TS Z/2$.
If $(Z_1,Z)$ is a trajectory of \r{NIcP7}, then
\brs
\frac{d}{dt}W(Z_1,Z)&=&-Z_1^T ((\bar k_1 C_\star -\Pi_1\Om) Z_1-\Pi_1 Z)
-Z^T S ( ({\cal J}+\Om\Pi_1) \Om Z_1-({\cal J}+\Om\Pi_1)Z)\\
&\leq&Z_1^T (-\bar k_1 C_\star +\Pi_1\Om) Z_1-\lb \|Z\|^2+(\|\Pi_1\|+\|S({\cal J}+\Om\Pi_1)\Om\|)\|Z_1\|\|Z\|\\
&\leq&(-\bar k_1 \xi +\delta_1) \|Z_1\|^2 -\lb \|Z\|^2
+\delta_2 \|Z_1\|\|Z\|,
\ers
where the constants $\delta_1,\delta_2>0$ do not depend on $\bar k_1$.
Since
$$\|Z_1\|\|Z\|\leq  \eps^2 \|Z_1\|^2+\frac{\|Z\|^2}{\eps^2}$$
for every $\eps>0$, then
\brs
\frac{d}{dt}W(Z_1,Z)
&\leq&\lp -\bar k_1 \xi +\delta_1 +\frac{\delta_2}{\eps^2}\rp \|Z_1\|^2+(-\lb+\eps^2 \delta_2)  \|Z\|^2.
\ers
Choosing $\eps^2$ small enough in order to have   $-\lb+\eps^2 \delta_2\leq -\lb/2$ and $\bar k_1$ large enough, we have
$$
\frac{d}{dt}W(Z_1,Z)
\leq - \frac \lb 2  (\|Z_1\|^2+ \|Z\|^2).
$$

The proof is concluded, since $(Z_1,Z)$ and $(Y_1,Y)$ are equivalent
systems of coordinates on the space $E^h_{m_1,\dots,m_{r_1+1}}$.
\end{proof}

The proof of
Theorem~\ref{thm2} is 
completed by applying the same contradiction argument as in 
the proof of Theorem~\ref{theo}. 
\end{proof}

\section{Maximal rates of exponential convergence and
divergence}\label{s_maximal-rate}

Let $(A,b)\in M_n(\R)\times \R^n$ be a controllable pair, $K$ belong to $\R^n$
and $T,\mu$ be positive constants such that $T\geq \mu$.
For $\al\in\G(T,\mu)$ let
$\lb^+(\al,K)$ and  $\lb^-(\al,K)$ be, respectively, the maximal and minimal
Lyapunov exponents associated with $\dot x=(A-\al b K^T)x$,
i.e.,
$$
\lb^+(\al,K)= \sup_{\|x_0\|=1}\limsup_{t\to+\infty} \frac{\log(\|x(t;0,x_0,K,\al)\|)}t,\ \ \
\lb^-(\al,K)=\inf_{\|x_0\|=1}\liminf_{t\to+\infty}  \frac{\log(\|x(t;0,x_0,K,\al)\|)}t.
$$

The {\it rate of convergence} (respectively,  the {\it rate of divergence}) associated with the family of systems
$\dot x=(A-\al b K^T)x$, $\al\in \G(T,\mu)$, is defined as
\be\label{td0}
\td(A,b,T,\mu,K)=-\sup_{\al\in
\G(T,\mu)}\lb^+(\al,K)\ \
\ (\hbox{respectively, } \tdD(A,b,T,\mu,K)=\inf_{\al\in \G(T,\mu)}\lb^-(\al,K)).
\ee
 Notice that
\be\label{td-vp}
\td(A,b,T,\mu,K)\leq \min_{\bar\al\in[\mu/T,1]}\min\{-\Re(\sigma(A-\bar\al b
 K^T))\},\ee
 and
$$
\tdD(A,b,T,\mu,K)\leq \min_{\bar\al\in[\mu/T,1]}
\min\{\Re(\sigma(A-\bar\al b
 K^T))\}.
 $$

Moreover, since a linear change of coordinates $x'=P x$ does not
affect Lyapunov exponents, then
\be\label{chco1} \td(A,b,T,\mu,K)=\td(P A P^{-1},P b,T,\mu,(P^{-1})^T
K), \ee
and
\be\label{chco}
\tdD(A,b,T,\mu,K)=\tdD(P A P^{-1},P b,T,\mu,(P^{-1})^T K).\ee

Define the maximal rate of convergence associated with the PE system
$\dot x=Ax+\al bu$, $\al\in \G(T,\mu)$, as
\be\label{tdd}
\tdd(A,T,\mu)=\sup_{K\in\R^n} \td(A,b,T,\mu,K),
\ee
and similarly, the
maximal rate of divergence as
\be\label{tDD}
\tDD(A,T,\mu)=\sup_{K\in \R^n}\tdD(A,b,T,\mu,K).
\ee
Notice that neither
$\tdd(A,T,\mu)$ nor $\tDD(A,T,\mu)$ depend on $b$, as it
follows from \r{chco1} and \r{chco}.

\begin{rmk}\label{rem2}
Let us collect  some properties of $\tdd$ and $\tDD$ that follow
directly from their definition. First of all, one has \be\label{ttt}
\tdd(A+\lb\Id_n,T,\mu)=\tdd(A,T,\mu)-\lb, \ \
\tDD(A+\lb\Id_n,T,\mu)=\tDD(A,T,\mu)+\lb . \ee 
Then, by time-rescaling, 
\be\label{39+}
\tdd(A,T,\rho T)=\tdd(A/T,1,\rho), \ \ \ \tDD(A,T,\rho
T)=\tDD(A/T,1,\rho).
\ee
Notice moreover that, thanks to (\ref{elementary}), both
$\tdd(J_n,T,\rho T)$ and $\tDD(J_n,T,\rho T)$ only depend on $\rho$
and thus are equal  to $\tdd(J_n,1,\rho)$ and
$\tDD(J_n,1,\rho)$, respectively. Finally, because of point~\ref{rho202} in
Lemma~\ref{proprieta}, $\tdd$ and $\tDD$ are monotone with respect
to their third argument.
\end{rmk}

\begin{rmk}\label{rem3}
Given a controllable pair $(A,b)$ and a class ${\cal{G}}(T,\mu)$ of
PE signals,
whether or not $\tdd$ and $\tDD$ are both infinite can
be understood
as whether or not a pole-shifting
type property holds true for the PE control system $\dot
x=Ax+\al bu$, $\al\in \G(T,\mu)$.
\end{rmk}

The study
of the pole-shifting type property for two-dimensional PE systems
actually reduces to that of their maximal rates of convergence as a consequence
of the following property.

\begin{prop}\label{rdc0}
Consider the two-dimensional PE systems $\dot
x=Ax+\al bu$, $\al\in\caG(T,\mu)$, with $(A,b)$ controllable.
Then $\tdd(A,T,\mu)=+\infty$ if and only if $\tDD(A,T,\mu)=+\infty$.
\end{prop}
\begin{proof}
According to \rref{chco1}, \rref{chco} and \rref{ttt}, it is enough to prove the result
for $(A,b)$ in companion form and with $\Tr(A)=0$. Let then
\be\label{comp-form}
A=\lp \ba{cc}0&1\\a&0\ea\rp\ \ \ \ Êb=\lp\ba{c}0 \\1\ea\rp,
\ee
with $a\in\R$.

Assume that $\tdd(A,T,\mu)=+\infty$.
By definition, for every $C>0$ there exists $K\in\R^2$ such that $\td(A,b,T,\mu,k)>C$.
Therefore, by definition of $\td$,
\be\label{oneway}
\limsup_{t\to+\infty} \frac{\log(\|x(t;0,x_0,K,\al)\|)}t<-C,\ \ \ \ \ \forall \al\in\G(T,\mu), \forall \|x_0\|=1.
\ee
Moreover, due to \r{td-vp},
for $C$ large enough we can assume that
$k_1,k_2$ and $k_1/k_2$ are large positive
numbers.

Let $K_-=(k_1,-k_2)$. We claim that if $C$ is large enough then
$\tDD(A,b,T,\mu,K_-)\geq C$. Assume by contradiction that there
exists $\bar \al\in \G(T,\mu)$ such that $\lb^-(\bar \al,K_-)<C$.
Then there exists $\bar x\in\R^2$ of norm one and an increasing
sequence $(t_n)_{n\in\N}$ of positive times going to infinity such
that
$$
\frac{\log(\|x(t_n;0,\bar x,K_-,\bar \al)\|)}{t_n}<C,\ \ \
\forall\in\N.
$$
Notice that for
every $t\in [0,t_n]$,
$$
x(t;0,\bar x,K_-,\bar \al(\cdot))=\diag(1,-1)x(t_n-t;0,x_n,K,\bar \al(t_n-\cdot)),
$$
where $x_n=\diag(1,-1)x(t_n;0,\bar x,K_-,\bar\al)$.

Therefore, by homogeneity,
\be\label{alg--}
\frac{\log\left(\left\|x\left(t_n;0,\frac{x_n}{\|x_n\|},K,\bar
\al(t_n-\cdot)\right)\right\|\right)}{t_n}=-\frac{\log(\|x_n\|)}{t_n}=
-\frac{\log(\|x(t_n;0,\bar x,K_-,\bar \al)\|)}{t_n}>-C.
\ee

This would contradict \r{oneway} if, for some positive integer $n$,
$x_n/\|x_n\|=\bar x$ and the signal obtained by repeating
$\bar\al|_{[0,t_n)}$ by periodicity over $\Rp$ belonged to
$\G(T,\mu)$.
Indeed, in such a case,
\be\label{kvolte}
\frac{\log\left(\left\|x\left(k t_n;0,\bar x,K,\tilde
\al(\cdot)\right)\right\|\right)}{k t_n}
>-C
\ee
for every $k\geq 1$, where $\tilde \al\in \G(T,\mu)$ denotes the
signal obtained by repeating $\bar\al|_{[0,t_n)}(t_n-\cdot)$ by
periodicity over $\Rp$.

In order to recover the periodic case, we are going to extend $\bar
\al$ backwards in time over an interval $[-2\mu-\tau_n,0)$ as
follows. First set $A_1^-=A-bK_-^T$. We take
$\bar\al=1$ on the intervals $[-\mu,0)$ and
$[-2\mu-\tau_n,-\mu-\tau-n)$ and we extend $\bar\al$ on
$[-\mu-\tau_n,-\mu)$ in such a way that the trajectory corresponding
to $\bar \al|_{[-\mu-\tau_n,-\mu)}$ and to the gain $K_-$ connects the half-line $\Rp
x_n^+$ to $\bar x^-$, where $x_n^+=\exp(\mu A_1^-)\diag(1,-1)x_n$ and $\bar
x^-=\exp(-\mu A_1^-)\bar x$.
We show below that this can be done fulfilling  the PE condition and with
$\tau_n$ upper bounded 
by a  constant independent of $n$.
Hence,
the signal
obtained extending  $\bar \al_{[-2\mu-\tau_n,t_n]}$ by periodicity
belongs to
$\G(T,\mu)$ and we have
\brs
x\left(t_n+2\mu+\tau_n;0,x_n,K,\bar
\al(t_n+2\mu+\tau_n-\cdot)\right)&\in&\Rp x_n\\
\log\left(\left\|x\left(t_n+2\mu+\tau_n;0,\frac{x_n}{\|x_n\|},K,\bar
\al(t_n+2\mu+\tau_n-\cdot)\right)\right\|\right)&=&
\log\left(\|\tilde{x}\|\right) -
\log(\|x(t_n;0,\bar x,K_-,\bar \al)\|),
\ers
where
$\tilde{x}=x(\tau_n+2\mu;0,\diag(1,-1)\bar x,K,\bar\al|_{[-2\mu-\tau_n,0]}(-\,\cdot))$.
Note that $\log(\|\tilde{x}\|)$ can be
lower bounded independently of $n$, because of the uniform boundedness of $\tau_n$.
Therefore,
\brs
\frac{\log\left(\left\|x\left(t_n+2\mu+\tau_n;0,x_n,K,\bar
\al(t_n+2\mu+\tau_n-\cdot)\right)\right\|\right)}{t_n+2\mu+\tau_n}&>&
\frac{\log\left(\|\tilde{x}\|\right)}{t_n+2\mu+\tau_n} -\frac{C t_n}
{t_n+2\mu+\tau_n}
\ers
is larger than $-C$ for $n$ large enough and we can conclude as in \r{kvolte}.

We are left to prove that the control system on the unit circle
whose admissible velocities are 
the projections of the
linear vector fields $x\mt (A-\xi b K_-^T)x$, $\xi\in [0,1]$, is
completely controllable in finite time by controls $\xi=\xi(t)$ satisfying the
PE condition. Notice that the equilibria of the projection of a linear vector field $x\mt A'x$
on the unit circle 
are given by
the eigenvalues of $A'$. 
All other trajectories are heteroclinic connections between the equilibria, unless the eigenvalues of $A'$ are non-real, in which case the phase portrait is given
 by a single periodic trajectory.

Denote by $\theta$ a point on the unit circle, identified with $\R/2\pi\Z$. Then, the above
mentioned control system on the unit circle can be written
\be\label{unit0}
\dot\tht=a\cos^2(\tht)-\sin^2(\tht)+\xi\cos(\tht)\left(
k_2\sin(\tht)-k_1\cos(\tht)\right),\ \ \xi\in[0,1].
\ee

We prove the controllability of~\rref{unit0}
by exhibiting a trajectory $\bar\tht$ of
\rref{unit0} corresponding to a  PE control $\bar\xi$, starting at some $\tht_0\in \R/2\pi\Z$, making a complete turn and
going back in finite time to $\tht_0$.

The PE condition will be verified by checking that
the control $\bar\xi=0$ is applied for a total time that is
smaller than $T-\mu$.
Define the angle $\tht_K\in (0,\pi/2)$ by
$$\tan\lp \tht_K\rp=2\frac{k_2}{k_1}.$$
Notice that
the eigenvectors of $A_1^-$ are proportional to the vectors
$(2,k_2\pm \sqrt{k_2^2-4(k_1-a)})$. Therefore, assuming that $k_1$ is larger than $a$, the
angle  between any real eigenvector of $A_1^-$ and the vertical axis is smaller than
$\th_K$.

Take
$\th_0=\pi/2$
and apply
$\bar\xi= 0$ until
$\bar\tht$ reaches $\pi/2-\th_K$.
Since $k_2/k_1$ is small and $\th_K$ is of the same order as $k_2/k_1$, then we can assume that $a\cos^2(\tht)-\sin^2(\tht)<-1/2$ for
$\th\in [\pi/2-\tht_K,\pi/2]$. Therefore, the time needed to go from $\pi/2$ to $\pi/2-\tht_K$ can be assumed to be smaller than $(T-\mu)/2$.
When the trajectory $\bar\tht$ reaches $\pi/2-\th_K$, switch to $\bar\xi= 1$
and apply it
until
$\bar\tht$ reaches (in finite time) $-\pi/2$.
This is possible since either
the eigenvectors of $A_1^-$ are non-real
or 
they
are contained in the cone
$$\{(r\cos\th,r\sin\th)\mid r>0,\;\th\in (\pi/2-\th_K+m\pi, \pi/2+m\pi),\;m\in\Z\}.$$
In both cases the dynamics of \r{unit0} with $\xi=1$
describe a non-singular clockwise rotation on the
arc of the
unit circle corresponding to $[\pi/2,\pi/2-\th_K]$.
The trajectory is completed, by homogeneity,
taking $\bar\xi= 0$ until
$\bar\tht$ reaches $-\pi/2-\th_K$ and finally $\bar\xi= 1$
until $\bar\tht$ reaches $-3\pi/2=\pi/2\mbox{ $($mod }2\pi)$.
As required,
the sum of the lengths of the intervals on which $\bar\xi=0$ does not exceed
$T-\mu$.

This concludes the proof that $\tdd(A,T,\mu)=+\infty$ implies
$\tDD(A,T,\mu)=+\infty$.
The converse 
can be proven by a perfectly analogous argument.
\end{proof}

\subsection{Arbitrary rates of convergence and divergence for $\rho$ large enough}
This section aims at
proving
that for $\rho$ large enough a persistently
excited system can be either stabilized
with an
arbitrarily large rate of exponential convergence or
destabilized with an
arbitrarily large rate of exponential divergence. This will be done by adapting the classical
high-gain technique.

\begin{prop}\label{rho^*}
Let $n$ be a positive integer. There exists $\rho^*\in(0,1)$ 
such that for every controllable pair $(A,b)\in
M_n(\R)\times \R^n$, every $T>0$ and every $\rho\in(\rho^*,1]$ one has
$\tdd(A,T,\rho T)=\tDD(A,T,\rho T)=+\infty$.
\end{prop}
\begin{proof}
Fix $T>0$ and let $(A,b)\in M_n(\R)\times \R^n$ be a controllable pair
in companion form. According to \r{ttt}, it is enough to establish
the result with the extra hypothesis that $\Tr(A)=0$. We therefore
assume in the sequel that $b=(0,\dots,0,1)^T$, $A=J_n+b K_A^T$  and
$K_A^T b=0$.

We first prove the stabilization result. Fix $K\in\R^n$ such that $
J_n-b K^T$ is Hurwitz. Let $P$ be the unique positive definite $n\times n$
matrix that solves the Lyapunov equation
$$ (J_n-b K^T)^TP+P (J_n-b K^T)=-\Id_n.$$
Define $V(x)=x^T P x$.
Then,
for every $\al\in L^\infty(\R,[0,1])$ and every solution
of
$\dot x=(J_n-\alpha b K^T)x$,
one has
$$\frac{d}{dt}V(x(t))\leq -C_1 V(x(t))+ C_2(1-\al(t))V(x(t)),$$
with $C_1,C_2$ two positive constants only depending  on $K$.
Choose $\rho\in (0,1)$ and assume that $\al$ is a $(T,T\rho)$-signal.
Then, for every $t\geq0$,
$$V(x(t+T))\leq V(x(t))\exp(-T(C_1-C_2(1-\rho))).$$
Therefore, if $\rho>1-(C_1/2C_2)$ then
$\tdd(J_n,T,T\rho)\geq C_1/2>0$.
For every $\gamma>0$, set  $K_\ga=\ga D_\ga K$ (where, as in the previous section,
$D_\gamma=\diag(\gamma^{n-1},\dots,\gamma, 1)$).
Recall that $J_n$ and $D_\ga$
satisfy \r{elementary}.
Take a solution 
of $\dot x=(A-\alpha b K_\ga^T)x$ with
$\al\in\G(T,\rho T)$.   Set $z(\cdot)=D_\ga x(\cdot)$ and notice that for
every $\ga>1$
$$\frac{d}{dt}V(z(t))\leq \ga(-C_1 + C_2(1-\al(t))+C_A/\ga^2)V(z(t)),$$
where $C_A$ only depends on $K_A$ and $P$. Then clearly
$\tdd(A,T,T\rho)\geq \ga C_1/3$ for $\rho>1-(C_1/2C_2)$ and $\ga$
large enough. Thus, $\tdd(A,T,T\rho)=+\infty$ and one can choose
$\rho^*\geq 1-(C_1/2C_2)$.

The destabilization result can be obtained by a similar argument
based on the Lyapunov equation
$$ (J_n-b L^T)^TQ+Q (J_n-b L^T)=\Id_n,$$
verified for some $L\in\R^n$ and some symmetric positive definite matrix
$Q$.
\end{proof}

\subsection{Finite maximal rate of convergence for $\rho$ small enough}
In this section we restrict our attention to the case $n=2$.

\begin{prop}\label{rho_*}
 There exists $\rho_*\in(0,1)$  such that for every controllable pair
 $(A,b)\in M_2(\R)\times \R^2$,  every $T>0$ and every $\rho\in(0,\rho_*)$ one has
 $\tdd(A,T,\rho T)<+\infty$.
\end{prop}
\begin{proof}
Thanks to Remark~\ref{rem2}, it suffices to show that
 there exists $\rho_*\in(0,1)$  such that, for every controllable
 pair $(A,b)\in M_2(\R)\times \R^2$ with $\Tr(A)=0$, one has
 $\tdd(A,1,\rho_*)<+\infty$.

As in \r{comp-form}, take $(A,b)$ in companion form, ie,
$$A=J_2+a H,\  \  \ Êb=(0,1)^T,$$
with $a\in\R$ and $H=\lp \ba{cc}0&0\\1&0\ea\rp$.

For $\theta\in [-\pi,\pi)$ set $e_\th=(\sin\th,\cos\th)^T$ and define $y_0=(-1, \ 0)^T$.
Every gain can be written as
$$K_{\th,\ga}=\ga D_\ga e_\th,$$
with $\ga\geq 0$ and $\theta\in [-\pi,\pi)$.

Moreover, if  $A-b K^T$ is Hurwitz with $K=\ga D_\ga e_\th$
   then the sum and the product of its two eigenvalues are, respectively,
   $\ga\cos\th>0$ and $\ga^2 \sin\th-a>0$. In particular,
   $\th\in (-\pi/2,\pi/2)$ and $\ga^2 \sin\th>a$.
If $\th\in(-\pi/2,0]$ with $A-b K^T$ Hurwitz, then $|a-\sin\th
\ga^2|\leq |a|=-a$ and therefore the convergence rate of $A-b K^T$
is upper bounded by a constant only depending on $a$.

Let $\Omega_0=(0,\pi/2)\times (0,\infty)$. We show in the following
the existence of $\rho>0$ and 
$\Omega =\{(\th,\gamma)\mid 0<\th<\pi/2,\;0<\ga<\ga(\th)\}\subset \Omega_0 $ such that
\be\label{(1)} \mbox{ if $(\theta,\ga)\in \Om_0$ and $K_{\th,\ga}$
is a $(1,\rho)$-stabilizer of $\dot x=Ax+\al bu$, then
$(\th,\ga)\in\Omega,$ } \ee and \be\label{(2)}
\sup_{(\th,\ga)\in\Omega} \min\{-\Re(\sigma(A-b
K_{\th,\ga}^T))\}<+\infty, \ee and the conclusion then follows
from
\r{td-vp}.

%

Fix $\theta\in (0,\pi/2)$. In order to find, for $\ga$ large enough, $\al\in{\G} ( 1,\rho)$  and $x_0\in \R^2$ 
such that the trajectory of 
$$\dot x=Ax-\al b K_{\th,\ga}x,\   \   \  x(0)=x_0,$$
is unbounded, we
apply the transformation $y_\ga(\cdot)=D_\ga x(\cdot/\ga)$: the
problem is now to find, for  $\ga$ large enough, 
$\al\in\G(\ga,\rho\ga)$ and an unbounded trajectory of 
\be\label{ccc} \dot
y=\lp J_2+\frac a{\ga^2} H\rp y-\al b e_{\th}y. 
\ee

Due to the homogeneity of the system, the latter fact reduces to
determine $\tau$ large enough and $\al\in\G(\tau,2\rho\tau)$ such
that the solution $y(\cdot\,;0,y_0,e_\th, \al)$ of \r{ccc} satisfies
$y(\tau;0,y_0, \al)=-\xi y_0$ with $\xi>1$.  Indeed, for every
$\ga>\tau$ the extension of $\al|_{[0,\tau)}$ by periodicity is a
$(\ga,\rho\ga)$-signal (see point~\ref{rho2} in
Lemma~\ref{proprieta}) and the sequence $\|y(m\tau;0,y_0,
\al)\|=\xi^m$ goes to infinity as $m$ goes to infinity.

Set
$$M_\th=J_2-b e_\th^T,\ \ \  \    N_{a,\th,\ga}=J_2+\frac a{\ga^2} H-b e_\th^T.$$

Consider $h>0$ small to be fixed later. We distinguish two cases
depending on whether $\th\in (0,h)$ or not.

\medskip

{\bf The case $\th\in [h,\pi/2)$.}

\medskip

We construct a PE signal $\al$ as follows: starting at $y_0$ take
 $\al=1$ until the trajectory $y(\cdot\,;0,y_0,e_\th,\al)$ of
 \r{ccc}
  reaches, at time $T_1$,
 the switching line $\sin(\th) x+\cos(\th) y=0$.
In order to ensure that the switching line is reached in finite time
and, moreover,  that $T_1$ is lower and upper bounded by two
positive constants
  only depending on $h$ (and not on $\th\in[h,\pi/2)$),
it suffices  to choose  $\ga>\Ga_1(a,h)>0$ with $\Ga_1(a,h)$ only
depending on $a$ and $h$. (Indeed, the bounds hold for all matrices
in a neighborhood of $\{M_\th\mid \th\in [h,\pi/2)\}$ and it
suffices to ensure that $N_{a,\th,\ga}$ belongs to such
neighborhood.)

From  $y(T_1;0,y_0,e_\th,\al)$
  set $\al=0$ until the first coordinate of $y(\cdot\,;0,y_0,e_\th,\al)$ takes,
at time $T_1+T_2$, the value $1$.
    Finally,
 take $\al=1$
 until the second coordinate of $y(\cdot\,;0,y_0,e_\th,\al)$ reaches, at time
$T_1+T_2+T_3$, the value $0$.
(See Figure~\ref{miod}.)
\begin{figure}[h!]
\begin{center}
\input{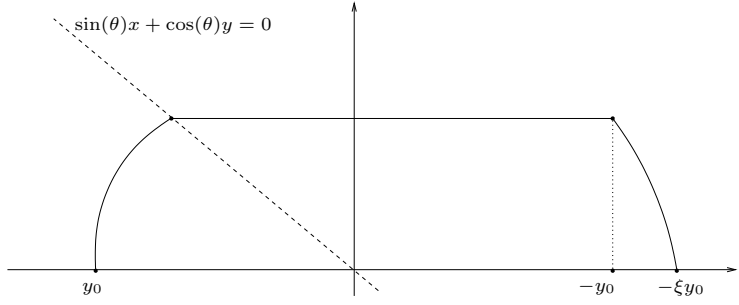}
\caption{\label{miod}The trajectory $y(\cdot\,;0,y_0,e_\th,\al)$ when
$\th\in [h,\pi/2)$}
\end{center}
\end{figure}

Analogously to what happens for $T_1$, the values $T_2$ and $T_3$ admit
lower and upper positive bounds
  only depending on $h$.

Define
$\tau=T_1+T_2+T_3$ and
notice that it admits an upper bound $\Tau_1(h)$ only depending on $h$. Finally,
$\frac{T_1+T_3}{T_1+T_2+T_3}$
admits a lower bound $\rho_1$ only depending on $h$.
The construction of the required $(\tau,\rho_1 \tau)$-signal
is achieved and we set
\be\label{gth1} \ga(\theta)\equiv \max(\Ga_1(a,h),\Tau_1(h)). \ee

\medskip

{\bf The case $\th\in (0,h)$.}

\medskip

Notice that the condition for $N_{a,\th,\ga}$ to be Hurwitz is that
$\ga^2>|a|/\sin\th$. Choose $\ga>\Ga_2(a,\th)=M\sqrt{|a|} /\sin\th$
with $M$ large (to be fixed later independently of all parameters).
In particular, for $M$ large enough and $h_0>0$  small enough
(independent of all parameters), for every  $\th\in(0,h_0)$ and
every $\ga>\Ga_2(a,\th)$ the matrix $N_{a,\th,\ga}$  has two real
eigenvalues, denoted by $\mu_+(a,\th,\ga)>\mu_-(a,\th,\ga)$ and
\be\label{borni} -2<\mu_-(a,\th,\ga)<-1/2,\ \ \
-2\sin\th<\mu_+(a,\th,\ga)<-\sin\th/2. \ee From now on we assume
$h\in(0,h_0)$.

Similarly to what has been done above, we construct a PE signal
$\al$ as follows: starting at $y_0$ take
 $\al=1$ in  \r{ccc}
for a time $T_1=\bar\rho M/|\mu_+(a,\th,\ga)|$ with $\bar\rho\in(0,1)$
to be fixed later.
Set $y_1=y(T_1;0,y_0,e_\th,\al)$.

 From $y_1$
set $\al=0$ for a time $T_2=M/|\mu_+(a,\th,\ga)|$ and denote by $y_2$
the point $y(T_1+T_2;0,y_0,e_\th,\al)$.    Finally,
 take $\al=1$
 until the second coordinate of $y(\cdot\,;0,y_0,e_\th,\al)$ assumes, at
 time $T_1+T_2+T_3$, the value $0$. (See Figure~\ref{caso2}.)

\begin{figure}[h!]
\begin{center}
\input{caso2bis.pstex_t}
\caption{\label{caso2}The trajectory $y(\cdot\,;0,y_0,e_\th,\al)$ when
$\th\in(0,h)$}
\end{center}
\end{figure}

We next show that there exist $\bar\rho$ and $M$ independent of $\th$
and $a$ such that $T_3$ is well defined and
$y(T_1+T_2+T_3;0,y_0,e_\th,\al)=-\xi y_0$ with $\xi>1$.

A simple computation yields
\brs
y_1&=&\frac1{\mu_-(a,\th,\ga)-\mu_+(a,\th,\ga)}\lp\ba{c}
e^{\mu_-(a,\th,\ga)T_1}\mu_+(a,\th,\ga)-e^{\mu_+(a,\th,\ga)T_1}\mu_-(a,\th,\ga)\\
\mu_-(a,\th,\ga)\mu_+(a,\th,\ga)(e^{\mu_-(a,\th,\ga)T_1}-e^{\mu_+(a,\th,\ga)T_1})
\ea\rp\\
&=&e^{-\bar\rho M}\lp\ba{c} -1\\ \mu_+(a,\th,\ga)\ea\rp +O(\th^2),
\ers
with $\|O(\th^2)\|\leq C\th^2$ and $C$ only depending on $M$ and $\bar\rho$.
(Similarly, in the sequel the symbol $O(\th)$ stands for a function of $\th$
upper bounded  by $C\th$ with $C$ only depending on $M$ and $\bar\rho$.)

In addition, one also gets that the first coordinate of $y_2$ is equal to
$$\left\{\ba{ll}
e^{-M\bar\rho}(M-1)+O(\th)&\mbox{if $a=0$,}\\
e^{-M\bar\rho}\lp M\frac{\mu_+(a,\th,\ga)}{\sin\th}\sinh\lp
\frac{\sin\th}{\mu_+(a,\th,\ga)}\rp-\cosh\lp \frac{\sin\th}{\mu_+(a,\th,\ga)}\rp\rp
+O(\th)&\mbox{if $a>0$,}\\
e^{-M\bar\rho}\lp M\frac{\mu_+(a,\th,\ga)}{\sin\th}\sin\lp
\frac{\sin\th}{\mu_+(a,\th,\ga)}\rp-\cos\lp
\frac{\sin\th}{\mu_+(a,\th,\ga)}\rp\rp+O(\th)&\mbox{if $a<0$.}
\ea
\right.
$$
Using \r{borni} one deduces that the first coordinate of $y_2$
is larger than
$$
\left\{\ba{ll}
e^{-M\bar\rho}( M/2\sinh(1/2)-\cosh(2))+O(\th)&\mbox{if $a>0$,}\\
e^{-M\bar\rho}( M/2\sin(1/2)-\cos(2))+O(\th)&\mbox{if $a<0$.}
\ea\right.
$$
Then in all three cases the first coordinate of $y_2$ becomes larger than
$$ e^{-M\bar\rho} (M C_0-C_1+O(\th)),$$
and one also gets that
 the second coordinate of $y_2$
can always be lower bounded by
$$ \sin\th e^{-M\bar\rho} (C_1-C_0/M+O(\th)),$$
with $C_0>0$ and $C_1>0$ independent of all the parameters.

Fix $M$ large and $\bar\rho\in(0,1)$ such that
$$
e^{-M\bar\rho} (M C_0-C_1)\geq 2,\ \ \ e^{-M\bar\rho} (C_1-C_0/M)\geq C_1/2.
$$

Finally, by eventually reducing $h$ in order to make each $O(\th)$
uniformly small, one can ensure that the first coordinate of $y_2$
remains larger than $1$ and that its second coordinate is positive.

Similar computations to the ones provided above show that it is
possible to further ensure that $T_3\leq 2 T_1$.

Define $\tau=T_1+T_2+T_3$. Then
$M/(2\sin\th)<\tau<8M/\sin\th=\Tau_2(\th)$. Choose now
\be\label{gth2}
\gamma(\theta)=M(8+\sqrt{a})/\sin\th\geq\max(\Tau_2(\th),\Gamma_2(a,\th)).
\ee

By construction, $\al\in\G(\tau,\bar\rho\tau)$. To conclude the
proof it is enough
 to check
condition \r{(2)} on
$$\Om_*=\{(\th,\gamma)\mid 0<\th<h,\;0<\ga< \ga(\th)\}.$$
For $(\th,\ga)\in \Om_*$
 define
 $$A^\mathrm{stab}_{\th,\gamma}=A- b K_{\ga,\th}^T =
 \lp\ba{cc} 0&1\\ a-\ga^2 \sin\th& -\ga\cos\th\ea\rp.$$
 Then
 $$0<\det(A^\mathrm{stab}_{\th,\gamma})\leq C_0 |\Tr(A^\mathrm{stab}_{\th,\gamma})|+|a|,$$
 with $C_0=2M(8+\sqrt{|a|})$, implying \r{(2)}.
\end{proof}

The following corollary is a direct consequence of Remark~\ref{rem2}
and Proposition~\ref{rho_*}.
\begin{cor}
Take $\rho_*$ as in the statement of Proposition~\ref{rho_*}. For
every controllable pair $(A,b)\in M_2(\R)\times \R^2$, every $T>0$
and every $\rho<\rho_*$, if $\lb>0$ is large enough, then
 $(A+\lb\Id_2,b)$ is not $(T,\rho T)$-stabilizable. Moreover, if
 $0<\rho<\rho_*$ and $\lb>\tdd(J_2,1,\rho)$, then
$(J_2+\lb\Id_2,b_0)$ is not $(T,\rho T)$-stabilizable for every
$T>0$.
\end{cor}
The above corollary establishes the existence of non-stabilizable PE
systems if the ratio $\rho=\mu/T>0$ is small enough and regardless of
$T$. This is rather intriguing when one recalls, on the one hand, that any
weak-$\star$ limit point $\al_\star$ of a sequence $(\al_{n})$, with
$\al_{n}\in \caG(T_n,\rho T_n)$ and $\lim_{n\rightarrow
  +\infty}T_n=0$, takes values in $[\rho,1]$ (see point~\ref{P1-1} of
Lemma~\ref{propri1}) and, on the other hand, that the switched system
$\dot x=J_2x+\al_\star(t) b_0u$, $\al_\star(t)\in [\rho,1]$, can be uniformly stabilized with an
arbitrary rate of convergence by taking the feedback law
$u_\gamma=-\gamma D_\gamma K x$, where $\gamma>0$ is arbitrarily large
and $K$ is provided by \cite[Lemma 4.0]{GK}.

\begin{rmk}
One possible interpretation of Proposition~\ref{rho_*} goes as
follows. Consider the destabilizing signals built in the
argument of the proposition back in the original time-scale, i.e.,
as $(1,\rho)$-signals. These signals take only the values $0,1$ over
time intervals of length
proportional to $1/\gamma$. Therefore, the
fundamental solution associated to $\dot x=(A-\al
b_0K_{\gamma,\th})x$ is a power of the product $A_1A_2A_3$, where
$A_1=\exp(T_1(A-b_0K_{\gamma,\th})/\gamma)$, $A_2=\exp(T_2 A/\gamma)$
 and
$A_3=\exp(T_3(A-b_0K_{\gamma,\th})/\gamma)$. 
The
stabilizing effect of $A-b_0K_{\gamma,\th}$ is countered by the
overshoot phenomenon occurring when the exponential of
$A-b_0K_{\gamma,\th}$ is taken only over small intervals of time.
If $\gamma$ is large enough,
such overshoot
eventually destabilizes $\dot x=(A-\al b_0K_{\gamma,\th})x$.
\end{rmk}

\subsection{Further discussion on the maximal rate of convergence}
Let $(A,b)\in M(n,\R)\times \R^n$ be a controllable pair.
 Define
\be\label{rho00}
\rho(A,T)=\inf\{\rho\in (0,1]\mid \tdd(A,T,T\rho)=+\infty\}.
\ee
Notice that $\rho(A,T)$ is equal to $\rho(A/T,1)$ and does not depend on $\Tr(A)$
(see Remark~\ref{rem2}).

Proposition~\ref{rho^*} implies that $\rho(A,T)\leq \rho^*$ for some
$\rho^*\in(0,1)$ only depending on $n$.  In the case $n=2$, moreover
Proposition~\ref{rho_*} establishes a uniform lower bound
$\rho(A,T)\geq \rho_*>0$.

The following lemma collects some further
properties of the function
$T\mapsto \rho(A,T)$. 

\begin{lemma}
Let $(A,b)\in M_n(\R)\times \R^n$ be a controllable pair. Then (i)
$T\mt \rho(A,T)$ is locally Lipschitz on $(0,+\infty)$; (ii) there exist
$\lim_{T\to+\infty}\rho(A,T)=\sup_{T>0}\rho(A,T)$ and $\lim_{T\to
0^+}\rho(A,T)=\inf_{T>0}\rho(A,T)$.
\end{lemma}
\begin{proof}
In order to prove  (i),
notice that point~{\ref{ligio}} in Lemma~\ref{proprieta} implies that
if $\tdd(A,T,\rho T)<+\infty$ then for every $\eta\in(0,\rho T)$,
\br
\tdd\lp A,T+\eta,\frac{\rho T}{T+\eta}(T+\eta)\rp&<&+\infty,\label{oone}\\
\tdd\lp A,T-\eta,
\frac{\rho T-\eta}{T-\eta}(T-\eta)\rp&<&+\infty.\label{ttwo}
\er
From \r{oone} we deduce that 
for every $\eta\in(0,\rho(A,T) T)$,
\be\label{mono}
\rho(A,T+\eta)\geq \frac{\rho(A,T) T}{T+\eta},
\ee
and thus 
$$\rho(A,T)-\rho(A,T+\eta)\leq \eta/T.$$

Similarly, \r{ttwo} implies that,  for every $\eta\in(0,\rho(A,T) T)$,
$$\rho(A,T-\eta)\geq \frac{\rho(A,T) T-\eta}{T-\eta}.$$
Therefore, one has 
\be\label{fnin}
\rho(A,T)\geq \frac{\rho(A,T+\eta)(T+\eta)-\eta}T
\ee
for every $\eta$ satisfying $0<\eta<\rho(A,T+\eta) (T+\eta))$ and in particular for every $\eta\in (0,\rho(A,T) T)$ (see \r{mono}).
We obtain from \r{fnin} that $\rho(A,T+\eta)-\rho(A,T)\leq \eta/T$ and we conclude that
$$|\rho(A,T+\eta)-\rho(A,T)|\leq \frac \eta T$$
for every $\eta\in (0,\rho(A,T) T)$.

As for point (ii), it suffices to deduce from point~{\ref{rhon}} in
Lemma~\ref{proprieta} that if $0<\rho'<\rho<1$ then there exists $M>0$
such that
whenever $\tdd(A,T,\rho T)=+\infty$ one has $\tdd(A,\ga,\rho'
\ga)=+\infty$ for every $\ga>0$ such that $\ga/T>M$.
\end{proof}

\begin{rmk}
In the case $A=J_n$ equality \r{Nro} implies that the function $T\mt \rho(J_n,T)$ is constant. When $n=2$ its constant value is positive, due to Proposition~\ref{rho_*}.
\end{rmk}

\section{Open problems}\label{OP}
We conclude the paper by providing some questions 
that arose from our investigation of single-input persistently excited linear systems.  

\begin{conjecture}\label{POLE}
Does Proposition~\ref{rdc0} still hold true in dimension bigger than
two? Notice that the  proof provided here
essentially relies on the controllability of \rref{unit0} in finite
time.
\end{conjecture}

\begin{conjecture}\label{inf0}
Consider the constant $\rho_n^*$ defined as the upper lower bound  
for all the $\rho^*$'s 
satisfying  the statement of Proposition~\ref{rho^*} ($n$ fixed).
What can be said on 
the dependance of $\rho_n^*$
on $n$ as $n\to\infty$?
\end{conjecture}

\begin{conjecture}\label{ndim}
We conjecture that Proposition~\ref{rho_*} holds true in
dimension $n>2$. Note however that the proof given in the 2D case
cannot be easily extended to the case in which $n>2$. Indeed, our
strategy is based on a complete parameterization of the candidate
feedbacks for stabilization and on the explicit construction of a
destabilizing signal $\al$ for every value of the parameter $\th$,
which 
takes values in 
the one-dimensional sphere. In the
general case, the parameter would belong to an $(n-1)$-dimensional
manifold and an explicit construction, if possible, would be more
intricate.
\end{conjecture}

\begin{conjecture}\label{waw}
It is a challenging question to 
determine whether the
function $T\mapsto \rho(A,T)$ (defined in \r{rho00}) is constant for a general matrix $A$.
If this is true, one may wonder whether its constant value
depends on $A$. Otherwise, a natural question would be
to understand the dependence of
$\lim_{T\to 0^+}\rho(A,T)$ and $\lim_{T\to+\infty}\rho(A,T)$ on the matrix $A$.
\end{conjecture}


 \begin{conjecture}\label{last1}
Proposition~\ref{rho_*}
states that, for $n=2$ and $\mu/T$ small, the PE control system
$\dot x=Ax+\al bu$, $\al\in \G(T,\mu)$, does not have the
pole-shifting property  (see Remark~\ref{rem3}). 
 It makes therefore sense to
investigate additional conditions to impose on the PE signals
(periodicity, positive dwell-time, uniform bounds on the derivative of
the PE signal, etc) so that the pole-shifting property holds true for
these restricted classes of PE signals, regardless of the ratio
$\mu/T$.  
First of all, the  subclass of
periodic PE signals must be excluded, since the destabilizing inputs constructed in
Proposition~\ref{rho_*} are periodic.
 It is also clear that, for the subclass of
 ${\cal{G}}(T,\mu)$ given by all signals with a positive dwell time $t_d>0$, one gets
 arbitrary rate of convergence (or divergence) with a linear constant feedback, for
 every choice of $T,\mu,t_d$. 
 Here follows our 
 conjecture.

   Given $T,M>0$ and $\rho\in (0,1]$, let ${\cal{D}}(T,\rho, M)$ be
   the subset of $\G(T,\rho T)$ whose signals are globally Lipschitz
   over $[0,+\infty)$ with Lipschitz constant bounded by $M$. Then,
   given a controllable pair $(A,b)$,
   we
   conjecture that it is possible to stabilize (respectively, destabilize) by a linear feedback
   the system $\dot x=Ax+\al b u$, $\al\in {\cal{D}}(T,\rho, M)$, with an arbitrarily large rate of convergence (respectively, divergence), i.e., we conjecture that for every $C>0$ there exist two gains $K_1$ and $K_2$ such that for every $\alpha\in {\cal{D}}(T,\rho, M)$
the maximal Lyapunov exponent of $\dot x=(A-\alpha bK_1^T)x$ is smaller than
$-C$ and the 
the minimal Lyapunov exponent of $\dot x=(A-\alpha bK_2^T)x$ is larger than
$C$.
 \end{conjecture}

\bibliographystyle{abbrv}
\bibliography{biblio_UC}

\end{document}